\documentclass[a4paper,11pt]{article}
\usepackage[utf8x]{inputenc}
\usepackage{a4wide}
\usepackage{amssymb}
\usepackage{amsmath}
\usepackage{amsthm}
\usepackage{latexsym}
\usepackage[mathscr]{euscript}
\usepackage{hyperref}
\usepackage[color]{changebar}
\usepackage{todonotes}

\theoremstyle{definition}
\newtheorem{definition}{Definition}[section]

\theoremstyle{plain}
\newtheorem{Theorem}[definition]{Theorem}
\newtheorem{Proposition}[definition]{Proposition}
\newtheorem{Lemma}[definition]{Lemma}

\newtheorem{Definition}[definition]{Definition}

\theoremstyle{remark}
\newtheorem{remark}[definition]{Remark}

\newcommand{\R}{\mathbb R}  
\newcommand{\N}{\mathbb N} 

\newcommand{\Vol}{\mathrm{Vol}}

\newcommand{\eps}{\varepsilon}
\newcommand{\vphi}{\varphi}
\cbcolor{blue}
\newcommand{\Ric}{\mathrm{Ric}}

\newcommand{\comp}{\Subset}
\newcommand{\X}{\mathfrak{X}}

\newcommand{\sse}{\subseteq}

\newcommand{\Om}{\Omega}
\newcommand{\Dpk}{\mathcal{D}'{}^{(k)}}
\newcommand{\Dpo}{\mathcal{D}'{}^{(1)}}
\newcommand{\trsm}{\mathcal{T}^r_s(M)}
\newcommand{\ltl}{L^2_{\mathrm{loc}}}
\newcommand{\D}{\mathcal{D}}
\newcommand{\cinfty}{C^\infty}
\newcommand{\lara}[2]{\langle #1, #2 \rangle}
\newcommand{\om}{\omega}
\newcommand{\lpl}{L^p_{\mathrm{loc}}}

\newcommand{\pac}{P^{\mathrm{ac}}_2}
\newcommand{\supp}{\mathrm{supp}}
\newcommand{\vol}{\mathrm{vol}}
\newcommand{\cJ}{\mathcal{J}}

\title{Synthetic versus distributional lower Ricci curvature bounds}

\author{Michael Kunzinger\footnote{University of Vienna, Faculty of Mathematics, 
michael.kunzinger@univie.ac.at}, 
Michael Oberguggenberger\footnote{University of Innsbruck, Unit of Engineering Mathematics, michael.oberguggenberger@uibk.ac.at},
James A.\ Vickers\footnote{University of Southampton, School of Mathematics, J.A.Vickers@soton.ac.uk}
}

\begin{document}

\date{}


\maketitle

\begin{abstract}
We compare two standard approaches to defining lower Ricci curvature bounds for Riemannian metrics of regularity below $C^2$.
These are, on the one hand, the synthetic definition via weak displacement convexity of entropy functionals in the 
framework of optimal transport, and the distributional one based on non-negativity of the Ricci-tensor in the sense of Schwartz.
It turns out that distributional bounds imply entropy bounds for metrics of class $C^1$ and that the converse holds for
$C^{1,1}$-metrics under an additional convergence condition on regularisations of the metric.

\vskip 1em

\medskip

\end{abstract}

\section{Introduction}

Applications of the theory of optimal transport to Riemannian geometry have had a transformative influence on the field and have led
to far-reaching generalisations of classical notions of curvature. In particular, lower Ricci-curvature bounds have been re-cast
into convexity conditions on certain entropy functionals, which make sense in settings that go far beyond the original field
of Riemannian geometry. Indeed, such curvature (or curvature-dimension) conditions in the sense of Lott-Villani and Sturm can be formulated in general metric measure spaces,
where no differentiable structure is available a priori (cf.\ \cite{LV09,Sturm1,Sturm2}, as well as the 
introductory texts \cite{Vil09,AG}). For Riemannian metrics of regularity $C^2$, 
these synthetic formulations of lower Ricci curvature bounds are known to be equivalent to the classical (pointwise) estimates 
(\cite{CMcCS,vRS}, see also \cite{Vil09}). 

In particular, synthetic Ricci curvature bounds can be imposed on Riemannian metrics whose regularity lies strictly below $C^2$. On the
other hand, for such metrics there is another, more analytic, way of making sense of curvature bounds, namely by calculating the 
curvature quantities directly in the space of Schwartz distributions and then imposing distributional inequalities as curvature bounds. 
Also this approach reduces to the classical pointwise estimates as soon as the regularity of the metric is at least $C^2$. Especially
in physics, be it in classical electrodynamics, general relativity, or quantum field theory, this analytic approach is widely used
to model singular sources or fields or to describe matched spacetimes (cf., e.g., \cite{BGP,SV} and references therein). In recent years, 
distributional Ricci bounds (in the shape of strong energy conditions) have also featured prominently in the generalisation of the 
classical singularity theorems of Penrose and Hawking to spacetime metrics of regularity below $C^2$ (\cite{penrosec11,hawkingc11,GGKS,G20,KOSS22}).
On the synthetic side, a generalisation of Hawking's singularity theorem to Lorentzian synthetic spaces was established by Cavalletti and Mondino in \cite{CM22}.

A natural question arising in this context is whether the synthetic and the distributional approach to lower Ricci curvature bounds
continue to agree for metrics of regularity strictly below $C^2$. We analyse this problem for Riemannian metrics of class $C^1$ and $C^{1,1}$
on compact manifolds, comparing distributional Ricci bounds with $\infty$-Ricci bounds in the sense of Lott-Villani (\cite{LV09,Sturm1,Sturm2}).
Our main results are that, on the one hand, distributional lower Ricci bounds imply lower $\infty$-Ricci bounds for $C^1$-metrics (Theorem \ref{th:dist_to_synth}),
and that, conversely, lower $\infty$-Ricci bounds on a $C^{1,1}$-Riemannian metric imply the corresponding distributional bounds
under an additional convergence condition on regularisations of the metric (Theorem \ref{th:synth_to_dist}). 

Technically, our approach
rests on a characterization of distributional curvature bounds via regularisation (Theorem \ref{th:Ricci_bounds_approx}), and on
a refined study of properties of the exponential map of a $C^{1,1}$-metric due to  Minguzzi (\cite{Min15}). The latter in particular
makes it possible to directly generalise a number of essential properties of optimal transport on Riemannian manifolds (as laid out in McCann's fundamental work
\cite{McC01}) from the regularity class $C^2$ to $C^{1,1}$. This is the content of Section \ref{sec:fundamentals_opt_transp}. In Section
\ref{sec:dist_curv_bounds} we give a brief overview of distributional Riemannian geometry and curvature bounds in this setting. 
The regularisation results we require are derived in Section \ref{sec:dist_bounds_via_reg}. The remaining Sections \ref{sec:dist_to_synth}
and \ref{sec:synth_to_dist} are devoted to proving the main results stated above.

\section{Fundamentals of optimal transport for $C^{1,1}$-metrics}\label{sec:fundamentals_opt_transp}
In this section we closely follow the fundamental paper \cite{McC01} and show that its results carry over unchanged from Riemannian metrics
of regularity $C^2$ to those of class $C^{1,1}$. Let $M$ be a compact connected
Riemannian manifold without boundary. We shall assume $M$ to be $C^\infty$-smooth, but all results derived below hold for $C^{2,1}$-manifolds 
as well. Indeed, by \cite[Th.\ 2.9]{H76}, for any $C^k$-manifold ($k\ge 1$) there exists a unique $C^k$-compatible $C^\infty$-atlas on $M$.
We are studying the Monge problem for probability measures $\mu, \nu$ on $M$ with cost function $c(x,y)=d^2(x,y)/2$, where $d$ is 
the Riemannian distance on $M$. Thus we are looking for a map $S:M\to M$ that minimises the transportation cost
\begin{equation}\label{eq:transp_cost}
\mathcal{C}(S) = \int_M c(x,S(x))\,d\mu(x)
\end{equation}
among all Borel maps $S$ that push forward $\mu$ to $\nu$, $S_\# \mu = \nu$, and we call the set of these maps $\mathcal{S}(\mu,\nu)$. 
For any $\psi: M\to \R\cup \{\pm\infty\}$ we define
its \emph{infimal convolution} by 
\[
\psi^c(y) := \inf_{x\in M} \big(c(x,y) -\psi(x) \big).
\]
The dual Kantorovich problem consists in finding $(\psi,\phi)$ maximising
\begin{equation}\label{eq:Kantorovich}
J(\psi,\phi) = \int_M \psi(x)\,d\mu(x) + \int_M \phi(y)\,d\nu(y)
\end{equation}
over the set $\mathrm{Lip}_c := \{(u,v) \mid u,v:M\to \R \ \mathrm{continuous},\ \forall (x,y):\ u(x)+v(y) \le c(x,y)\}$.

As a preparation for the following results, we collect some general properties of the exponential map and 
the distance function of a $C^{1,1}$-Riemannian metric. 
By \cite[Th.\ 4]{Min15} (or also \cite[Th.\ 4.1]{KSS14}), any $x\in M$ possesses a convex normal neighbourhood $U$ (called a totally
normal neighbourhood in \cite{KSS14}). This means that for any $y\in U$, $U$ is a normal neighbourhood of $y$, i.e., $\exp_y$ is a bi-Lipschitz
homeomorphism from a star-shaped open neighbourhood of $0$ in $T_yM$ onto $U$. In particular, Rademacher's theorem (cf.\ \cite[Lem.\ 4]{McC01} 
for the version we use here) implies that $\exp_y^{-1}$ is differentiable almost everywhere on its domain, as is $\exp_y$. 
By \cite[Th.\ 6]{Min15}, the unique shortest absolutely continuous curve $\sigma$ in $M$ from $y\in U$ to $x$ is (has a reparametrization as) the 
radial geodesic $\sigma:[0,1]\to M$, $\sigma(t) = \exp_y(t\cdot \exp_y^{-1}x)$. In particular, for the Riemannian distance of $x$ and $y$
we obtain  
\begin{equation}\label{eq:distance_exp}
d(x,y)= |\exp_y^{-1}x|_y.
\end{equation}
Covering a minimising curve between points that are not necessarily contained in a common convex normal neighbourhood
by such neighbourhoods and applying the above we conclude that also in this case the curve has a reparametrization as an unbroken geodesic,
hence in particular of regularity $C^{2,1}$ (as follows directly from the geodesic equation).

Furthermore, for any $x\in M$, $v\mapsto \exp_x v$ is strongly differentiable (in the sense of Peano, cf.\ \cite[Def.\ 2]{Min15}) at $v=0$ with 
strong differential $T_0\exp_x = \mathrm{id}_{T_xM}$ (\cite[Th.\ 3 and Sec.\ 2.3]{Min15}). By Leach's inverse function theorem (cf.\ \cite[Th.\ 2]{Min15}),
also $z\mapsto\exp_x^{-1}(z)$ is strongly differentiable at $z=x$ with strong differential $\mathrm{id}_{T_xM}$. In particular, the corresponding
statements with standard differentials are valid as well.  
Also the Gauss Lemma holds at every point $x\in U$ where $\exp_y$ is differentiable (\cite[Th.\ 5]{Min15}): for such $x$,
$v_1 := \exp_y^{-1}(x)$ and any $v_2\in T_yM\cong T_{v_1}(T_yM)$ we have:
\begin{equation*}\label{eq:gauss}
g(T_{v_1}\exp_y(v_1),T_{v_1}\exp_y(v_2)) = g(v_1,v_2).
\end{equation*}
Finally, it is proved in (\cite[Th.\ 5]{Min15}) that the map $D_y^2:=x\mapsto g(\exp_y^{-1}x,\exp_y^{-1}x)$ is of differentiability class 
$C^{1,1}$ on $U$, with tangent map 
\begin{equation}\label{eq:D_tangent}
T_x D^2_y = 2g(\dot\sigma(1),\,.\,),
\end{equation}
where $\sigma(t)=\exp_y(t\cdot \exp_y^{-1}x)$ is as above, and 
$(y,x)\mapsto P(y,x):=\dot\sigma(1)$ is the position vector field of $x$ with respect to $y$.

Based on these results we can now extend the validity of \cite[Prop.\ 6]{McC01} to Riemannian metrics $g\in C^{1,1}$:
\begin{Proposition}\label{prop:superdifferentiable}
Let $(M,g)$ be a $C^{1,1}$-Riemannian manifold (with or without boundary). 
Let $y\in M$ and set $\phi: M\to \R$, $\phi(x):= d^2(x,y)/2$. Then:
\begin{itemize}
\item[(i)] There exists a neighbourhood $U$ of $y$ such that $\phi$ is differentiable at 
every point in $U$.
\item[(ii)] For each $x\in M$, if there exists a distance-realizing absolutely continuous curve
$\sigma$ from $y$ to $x$, then $\sigma$ can be parametrized as
a geodesic $\sigma: [0,1]\to M$ and $\phi$ has supergradient $\dot\sigma(1)\in \bar\partial \phi_x$ at $x$. 
\end{itemize}
\end{Proposition}
\begin{proof} (i) Pick a convex normal neighbourhood $U\subseteq M$ of $y$. Then for any $x\in U$, by \eqref{eq:distance_exp} we have
$\phi(x) = \frac{1}{2} D^2_y(x)$, so \eqref{eq:D_tangent} gives $\nabla\phi(x) = \dot\sigma(1)$.
Using $d(x,y) = \sqrt{2\phi(x)}$ we conclude that $\nabla_x d(x,y) = \dot \sigma(1)/|\dot\sigma(1)|_x$
for $x\not=y$.

(ii) This now follows exactly as in \cite[Prop.\ 6]{McC01}. For the reader's convenience we include the argument. Suppose that $y$ does not 
lie in a convex normal neighbourhood $U$ around $x$ and let $\sigma$ be a minimising a.c.\  curve from $y$ to $x$. By what was said above
$\sigma$ possesses a parametrization as an unbroken geodesic $[0,1]\to M$ with $y=\sigma(0)$ and $x=\sigma(1)$. Pick $z\ne x$ on $\sigma$
such that  $z\in U$. Then (i) gives $\nabla_x d(x,z) = \dot \sigma(1)/|\dot\sigma(1)|_x$. Since $T_0\exp_x = \mathrm{id}_{T_xM}$, for any $v\in T_xM$ 
we have
\[
T_0(v\mapsto d(\exp_x v,z)) = g(\dot\sigma(1),\,.\,)/|\dot\sigma(1)|_x,
\]
which combined with the triangle inequality and the fact that $z$ lies on $\sigma$ leads to
\begin{align*}
d(y,\exp_x v) &\le d(y,z) + d(z,x) + g(\dot\sigma(1),v)/|\dot\sigma(1)|_x + o(|v|_x)\\
&= d(y,x) + g(\dot\sigma(1),v)/|\dot\sigma(1)|_x + o(|v|_x).
\end{align*}
This shows that $d(y,\,.\,) = \sqrt{2\phi}$ is superdifferentiable at $x$ and the chain rule \cite[Lem.\ 5]{McC01} finally
gives $\dot\sigma(1) \in \bar\partial \phi_x$, as claimed.
\end{proof}
Proposition \ref{prop:superdifferentiable} is the key to transferring all further results from \cite[Sec.\ 3]{McC01} to $C^{1,1}$-metrics.
We begin with \cite[Lem.\ 7]{McC01}:
\begin{Lemma}\label{lem:subdiff} Let $M$ be a connected closed Riemannian manifold with $g\in C^{1,1}$ and let $\psi: M\to \R$, $\psi=\psi^{cc}$. Then
$c(x,y) -\psi(x) -\psi^c(y) \ge 0$ for all $x,y\in M$, and if $\psi$ is differentiable in $x$ then equality holds if and only if
$y=\exp_x(-\nabla\psi(x))$.
\end{Lemma}
\begin{proof}
This follows exactly as in \cite[Lem.\ 7]{McC01}, only noting that the Hopf-Rinow theorem remains true in this regularity (indeed
even for $g\in C^0$, cf.\ \cite[2.5.28]{BBI01}) and the fact that minimisers are geodesics, hence $C^{2,1}$ by what was said above.
\end{proof}
Note that any $\psi$ as in the previous Lemma is Lipschitz continuous, hence differentiable a.e.\ by \cite[Lemma 2]{McC01}.
Based on the above results, the proofs of Theorems 8 and 9 in \cite{McC01} carry over verbatim to the current situation to give:
\begin{Theorem}\label{th:opt_maps} Let $M$ be a connected closed manifold with a $C^{1,1}$-Riemannian metric $g$ and let 
$\mu, \nu$ be Borel probability measures on $M$ with $\mu \ll \mathrm{vol_g}$. Also, let $c(x,y) = d(x,y)^2/2$. Then:
\begin{itemize}
\item[(i)] (Uniqueness) If $\psi: M\to \R$ satisfies $\psi=\psi^{cc}$ then $T: x\mapsto \exp_x(-\nabla\psi(x))$ minimises \eqref{eq:transp_cost}
among all Borel maps $S$ with $S_\#\mu = T_\# \mu$. Any other such map must equal $T$ $\mu$-almost everywhere.
\item[(ii)] (Existence) There exists a potential $\psi: M\to \R$ with $\psi=\psi^{cc}$ such that $T: x\mapsto \exp_x(-\nabla\psi(x))$ satisfies
$T_\# \mu = \nu$. Any other potential that pushes $\mu$ forward to $\nu$ gives rise to the same map $T$, up to a set of $\mu$-measure $0$.
\end{itemize}
\end{Theorem}
By Kantorovich duality (\cite[Th.\ 5.10]{Vil09}) it follows that for $\psi$ and $T$ as in Theorem \ref{th:opt_maps} we have
\begin{align*}
J(\psi,\psi^c) &= \int_M \psi(x)\, d\mu(x) + \int_M \psi^c(y)\,d\nu(y) = \sup_{(u,v)\in \mathrm{Lip}_c} J(u,v) \\
&= \min_{S\in \mathcal{S}(\mu,\nu)} \int_M c(x,S(x))\,d\mu(x) = \int_M c(x,T(x))\,d\mu(x)
= \min_{\pi\in \Pi(\mu,\nu)} \int_{M\times M} c(x,y)\,d\pi(x,y),
\end{align*}
where $\Pi(\mu,\nu)$ denotes the set of all couplings between $\mu$ and $\nu$.

We note that also invertibility of $T$ (\cite[Cor.\ 10]{McC01}) and polar factorization of maps (\cite[Th.\ 11]{McC01}) carry
over unchanged to $C^{1,1}$-metrics.
\section{Distributional curvature quantities}\label{sec:dist_curv_bounds}
In order to lay out the distributional approach to curvature bounds for metrics of low regularity within a consistent framework, we are going to employ
the theory of distributional connections due to LeFloch and Mardare (\cite{LeFMar07}, cf.\ also \cite{Mar68,GKOS,S08}), which we briefly summarize below.

For $k\in \N_0\cup \{\infty\}$ 
let $\mathrm{Vol}(M)$ denote the volume bundle over $M$, and $\Gamma^k_c(M,\mathrm{Vol}(M))$ the 
space of compactly supported $C^k$ one-densities on $M$ (i.e., sections of $\mathrm{Vol}(M)$) that are $k$ times continuously differentiable. Then the space
of distributions of order $k$ on $M$ is the topological dual of $\Gamma^k_c(M,\mathrm{Vol}(M))$ (cf.\ \cite[Sec.\ 3.1]{GKOS}):
\[
\Dpk(M) := \Gamma^k_c(M,\mathrm{Vol}(M))'.
\]
For $k=\infty$ we will omit superscript $(k)$. There are topological embeddings $\Dpk(M) \hookrightarrow {\D'}{}^{(k+1)}(M)
\hookrightarrow \D'(M)$ for all $k$. 
The space of distributional $(r,s)$-tensor fields of order $k$ is defined as
\[
\Dpk\mathcal{T}^r_s(M) \equiv \Dpk(M,T^r_s M) := \Gamma^k_c(M,T^s_r(M) \otimes \mathrm{Vol}(M))'. 
\]
Furthermore  (cf.\ \cite[3.1.15]{GKOS}), denoting by $\X(M)$ the space of smooth vector fields on $M$ and by
$\Om^1(M)$ that of smooth one-forms,
\begin{equation}\label{eq:Cinf-ext}
\D'\mathcal{T}^r_s(M) \cong \D'(M) \otimes_{C^\infty(M)} \mathcal{T}^r_s(M) \cong L_{C^\infty(M)}(\Omega^1(M)^r\times \X(M)^s; \D'(M)).
\end{equation}
and in fact these isomorphisms hold in the bornological sense (\cite{N13}).
Analogous isomorphisms hold in finite differentiability classes: For the $C^k(M)$-module $\Gamma_{C^k}(M,F)$
($1\le k \le \infty$) we have:
\begin{equation}\label{eq:Ck-ext}
\begin{split}
\Dpk\mathcal{T}^r_s(M) &\cong \Dpk(M) \otimes_{C^k(M)} (\mathcal{T}^r_s)_{C^k}(M)\\ 
&\cong L_{C^k(M)}(\Omega^1_{C^k}(M)^r\times \X_{C^k}(M)^s; \Dpk(M)).
\end{split}
\end{equation}

The space of smooth tensor fields is continuously and densely embedded via
\begin{align*}
\mathcal{T}^r_s(M) &\hookrightarrow   \Dpk\mathcal{T}^r_s(M) \\
t &\mapsto [(\theta_1,\dots,\theta_r,X_1,\dots,X_s) \mapsto [\omega \mapsto\int_M t(\theta_1,\dots,\theta_r,X_1,\dots,X_s)\omega]],
\end{align*}
where $\omega$ is a one-density. We note that the dense embedding of
$\mathcal{T}^r_s(M)$ in $\Dpk\mathcal{T}^r_s(M)$ already fixes the form of all the operations on distributional tensor fields
to be introduced below since we want to have compatibility with smooth Riemannian geometry.

Any $t\in \trsm$ possesses a unique extension to a map that acts on distributions in one of its arguments: 
if $\tilde\theta_1 \in \D'\mathcal{T}^0_1(M)$,
then since $t(\,.\,,\theta_2,\dots,X_s)\in \X(M)$ we may set
\begin{equation}\label{eq:dist_ext}
t(\tilde \theta_1,\theta_2,\dots,X_s) := \tilde\theta_1(t(\,.\,,\theta_2,\dots,X_s)) \in \D'(M),
\end{equation}
and analogously for the other arguments.

\begin{Definition}\label{def:dist_conn} 
A \emph{distributional connection} is a map 
\[
\nabla: \X(M)\times \X(M) \to \D'\mathcal{T}^1_0(M)
\]
such that for $X,X',Y,Y'\in \X(M)$ and $f\in C^\infty(M)$
the usual computational rules hold: $\nabla_{f X+X'}Y = f\nabla_XY + \nabla_{X'}Y$, $\nabla_X(Y+Y') = \nabla_X Y +\nabla_X Y'$, $\nabla_X(f Y) = X(f)Y + f\nabla_X Y$.
It is called an $\ltl$-connection, $\nabla\in \ltl$, if $\nabla_X Y$ is an $\ltl$-vector field for any $X,Y\in \X(M)$ (cf.\ \cite[Sec.\ 3]{LeFMar07}).
\end{Definition}
More generally, denoting by ${\mathcal G}$ any of the spaces $C^k$ $(0\leq k)$ or $\lpl$ $(1\leq p)$, we call a distributional connection 
a \emph{${\mathcal G}$-connection} if $\nabla_X Y$ is a ${\mathcal G}$-vector field for any $X,Y\in \X(M)$.  $\ltl$-connections 
play a distinguished role in this hierarchy since they form the largest class for which there is a stable definition of the curvature tensor in distributions, cf.\ \cite{GT87,LeFMar07,S08}.

Any $\ltl$-connection can be extended to a map $\nabla: \X(M) \times \X_{\ltl}(M) \to \D' \mathcal{T}^1_0(M)$ by setting
\begin{equation}\label{eq:L2loc_ext}
(\nabla_X Y)(\theta) := X(\theta(Y)) - Y(\nabla_X \theta) \qquad (X\in \X(M),\ Y \in \X_{\ltl}(M), \ \theta\in \Omega^1(M)).
\end{equation}
Based on this extension, we can assign a Riemann tensor to each $\ltl$-connection as follows (\cite[Def.\ 3.3]{LeFMar07}):
\begin{Definition}\label{def:dist_riem} The distributional Riemann tensor of an $\ltl$-connection $\nabla$ is the map $R: \X(M)^3 \to \D'\mathcal{T}^1_0(M)$,
\[
R(X,Y,Z)(\theta) := (\nabla_X\nabla_Y Z - \nabla_Y\nabla_X Z- \nabla_{[X,Y]} Z)(\theta) 
\]
for $X,Y,Z\in \X(M)$ and $\theta\in \Om^1(M)$,
\end{Definition}
If $F_i$ is a smooth local frame in $\X(U)$ and $F^j\in \Om^1(U)$ its dual frame (i.e., $F^j(F_i) = \delta_{i}^j$), then the Ricci tensor corresponding to $\nabla$ 
is given by (using the Einstein summation convention)
\begin{equation}\label{eq:Ric_def}
\Ric(X,Y) := (R(X,F_i)Y)(F^i) \in \D'(U) \qquad (X,Y\in \X(U)),
\end{equation}
and is readily seen not to depend on the choice of local frame.

A distributional Riemannian metric on $M$ (\cite{Mar68,GKOS,LeFMar07}) is an element of $\D'\mathcal{T}^0_2(M)$ that is symmetric and non-degenerate in the sense
that $g(X,Y)=0$ for all $Y\in \X(M)$ implies $X=0$. In particular, any $C^1$-Riemannian metric is an example of a distributional metric in this sense.
With a view to defining a metric connection in this generality, recall first the Koszul formula which uniquely determines the Levi-Civita connection 
of a smooth metric $g$ on $M$ (cf.\ \cite[Ch.\ 3]{ON83}):
\begin{equation}\label{eq:Koszul}
\begin{split}
2g(\nabla_X Y,Z) &= X(g(Y,Z)) + Y(g(Z,X)) -Z(g(X,Y)) \\
&- g(X,[Y,Z]) + g(Y,[Z,X]) + g(Z,[X,Y]) =: F(X,Y,Z)
\end{split}
\end{equation}
For a distributional metric $g$, we may use the right-hand side of \eqref{eq:Koszul} to define a bilinear map 
$\X(M)\times \X(M) \to \D'\mathcal{T}^0_1(M)$,
\begin{equation}\label{eq:nabla_flat}
\nabla^\flat_X Y := Z \mapsto \frac{1}{2} F(X,Y,Z) \in \D'\mathcal{T}^0_1(M),
\end{equation}
called the \emph{distributional Levi-Civita connection} of $g$ (\cite[Def.\ 4.2]{LeFMar07}). Note, however, that this is not (yet) a distributional connection
in the sense of Definition \ref{def:dist_conn} since it is of order $(0,1)$ instead of $(1,0)$. 
In addition to the standard product rules it also satisfies (the analogues of properties (D4), (D5) in \cite[Th.\ 3.11]{ON83})
\begin{equation}\label{eq:D4D5}
\begin{split}
\nabla^\flat_X Y - \nabla^\flat_Y X &= [X,Y]^\flat,\ \text{i.e.,} \ (\nabla^\flat_X Y - \nabla^\flat_Y X)(Z) = g([X,Y],Z)\\
X(g(Y,Z)) &= (\nabla^\flat_X Y)(Z) + (\nabla^\flat_X Z)(Y)
\end{split}
\end{equation}
for all $X, Y, Z\in \X(M)$. 

In order to obtain an $\ltl$-connection from $\nabla^\flat$ (which then will allow us to define the curvature tensors via Definition \ref{def:dist_riem}) 
we want to raise the index via $g$, i.e.,
\begin{equation}\label{eq:raise_index}
  g(\nabla_X Y,Z) := (\nabla^\flat_X Y)(Z) \qquad (X, Y, Z\in \X(M)).
\end{equation}
To be able to do this we need to restrict to metrics of higher regularity. It turns out that the \emph{Geroch-Traschen} class of metrics
is a reasonable family of metrics where this strategy can be implemented. This class consists of metrics in $H^1_{\mathrm{loc}}(M) \cap L^\infty_{\mathrm{loc}}(M)$
that are uniformly non-degenerate in the sense that $|\det g(F_i,F_j)|$ is locally bounded away from zero for any local frame $F_i$, see
\cite[Prop.\ 4.4]{LeFMar07}. For such metrics, hence in particular for any $C^1$-Riemannian metric, \eqref{eq:raise_index} defines an $\ltl$-connection
in the sense of Definition \ref{def:dist_conn}, which therefore has well-defined distributional curvature tensors.

Let us analyse the case of a $C^1$-Riemannian metric in more detail, as it will be the most relevant setting in this paper. For such a $g$ and $X,Y\in \X(M)$,
$\nabla_X Y$ is in fact a continuous vector field, so $\nabla$ is a $C^0$-connection, implying that
$(R(X,Y)Z)(\theta) \in \D'{}^{(1)}(M)$. Consequently, $R\in \Dpo\mathcal{T}^1_3(M)$.

Given $W,X,Y,Z \in \X(M)$ we then define $R(W,X,Y,Z)\in \Dpo$ by
\begin{align*}
R(W,X,Y,Z) := &X(g(W,\nabla_Y Z)) - Y(g(W,\nabla_X Z))\\ 
&- g(\nabla_X W,\nabla_Y Z) + g(\nabla_Y W,\nabla_X Z) - g( W,\nabla_{[X,Y]} Z).
\end{align*}
Then using \eqref{eq:dist_ext}, \eqref{eq:L2loc_ext} and \eqref{eq:D4D5} it follows that (cf.\ \cite[Rem.\ 4.5]{LeFMar07})
\[
R(W,Z,X,Y) = g(W,R(X,Y)Z).
\]
This identity can also be verified using local coordinates since there is a well-defined multiplication 
of distributions of first order with $C^1$-functions.

The Ricci tensor of a $C^1$-Riemannian metric $g$ is given by \eqref{eq:Ric_def}, and will be denoted by $\Ric_g$ or $\Ric(g)$. Alternatively, it can be calculated in terms of $g$-orthonormal
frames (which are only $C^1$): Note that \eqref{eq:L2loc_ext} remains valid even when
$Y$ and $\theta$ are only $C^1$, so $(\nabla_X(\nabla_Y Z))(\theta) \in \Dpo(M)$ ($X\in \X(M)$). Given a smooth local frame
$F_i$ with dual frame $F^j$, by \eqref{eq:Ric_def} we have $\Ric(X,Y) = (R(X,F_i)Y)(F^i)$. If now $E_i\in \X_{C^1}(U)$ is a local $g$-orthonormal
frame, $g(E_i,E_j)=\delta_{ij}$ and we set $E^i:=g(E_i,\,.\,) \in \Om^1_{C^1}(U)$. Expressing $F^i, F_i$ as $C^1$-linear combinations 
of $E^i, E_i$ then by what was said above we may calculate as in the smooth case to arrive at
\[
\Ric(X,Y) = \sum_i g(E_i, R(E_i,X)Y),
\]
where the scalar product is now between the $C^1$-vector field $E_i$ and the $\Dpo$-vector field $R(E_i,X)Y$. 
Finally, again by the above observations, the standard local formulae hold in $\Dpo$:
\begin{equation}\label{eq:Riem_Ricc_distr_local}
\begin{split}
R^m_{ijk} &= \partial_j\Gamma^m_{ik} - \partial_k\Gamma^m_{ij} + \Gamma^m_{js}\Gamma^s_{ik} - \Gamma^m_{ks}\Gamma^s_{ij} \\
\Ric_{ij} &= R^m_{imj}.
\end{split}
\end{equation}
Turning now to Ricci curvature bounds, recall (\cite[Ch.\ I, \S 4]{Sch66}) that a distribution $T\in \D'(U)$ ($U\subseteq \R^n$ open) is
called non-negative, $T\ge 0$, if $T(\vphi)\equiv\langle T,\varphi\rangle \ge 0$ for each test function $\varphi\ge 0$. In the manifold context we therefore
call $T\in \D'(M)$ non-negative if $T(\omega) \equiv \lara{T}{\om}\ge 0$ for any compactly supported non-negative one-density $\omega$. Any such distribution
is in fact of order $0$, hence is a measure on $M$. For $S, T\in \D'(M)$ we say that $S\ge T$ if $S-T\ge 0$. Following \cite[Def.\ 3.3]{G20}
we say that a $C^1$-Riemannian metric $g$ satisfies $\Ric \ge K$ for some $K\in \R$ if
\begin{equation}\label{eq:dist_Ric_bound}
\Ric(X,X) \ge K g(X,X) \qquad \forall X\in \X(M),
\end{equation}
where the inequality is in $\Dpo(M) \subset \D'(M)$, as explained above. Upper bounds are defined analogously.

In our analysis of distributional curvature bounds a main tool will be regularisation via convolution. To accomodate the 
manifold setting, we employ a construction from \cite[3.2.10]{GKOS}, \cite[Sec.\ 2]{KSSV}, \cite[Sec.\ 3.3]{G20}). 
Let $\rho\in \D(B_1(0))$ (with $B_1(0)$ the unit ball in $\R^n$), $\int \rho = 1$, $\rho\ge 0$ and, for $\eps\in (0,1]$, set $\rho_{\eps}(x):=\eps^{-n}\rho\left (\frac{x}{\eps}\right)$.
Pick a countable and locally finite family of relatively compact chart neighbourhoods 
$(U_i,\psi_i)$ ($i\in \N$), as well as a subordinate partition of unity $(\zeta_i)_i$ with $\mathrm{supp}(\zeta_i)\Subset U_i$ for all $i$.
Moreover, choose a family of cut-off functions $\chi_i\in\mathscr{D}(U_i)$ with $\chi_i\equiv 1$ on a
neighbourhood of $\mathrm{supp}(\zeta_i)$.
Denote by $f_*$ (respectively $f^*$) the push-forward (resp.\ pull-back) of a distribution under a diffeomorphism $f$, and set, for any $T \in \D'\mathcal{T}^r_s(M)$,
\begin{equation}\label{eq:M-convolution}
T\star_M \rho_\eps(x):= \sum\limits_i\chi_i(x)\,\psi_i^*\Big(\big(\psi_{i\,*} (\zeta_i\cdot T)\big)*\rho_\eps\Big)(x),
\end{equation}
where the convolution of tensor fields on open subsets of $\R^n$ is to be understood component-wise.  
Due to the presence of the cut-off functions $\chi_i$, the map $(\eps,x) \mapsto \mathcal{T}\star_M \rho_\eps(x)$
is smooth on $(0,1] \times M$. Note that, for any compact set $K\comp M$, there exists an $\eps_K$ such that for all $\eps<\eps_K$ and
all $x\in K$ \eqref{eq:M-convolution} is in fact a finite sum with all $\chi_i\equiv 1$. More precisely, this is the case 
whenever $\eps_K$ is less than the distance between the support of $\zeta_i\circ\psi_i^{-1}$ and the boundary of $\psi_i(U_i)$
for all $i$ with $U_i\cap K\neq \emptyset$.

This ``manifold convolution'' has smoothing properties that are closely analogous to those of convolution with a mollifier on 
open subsets of $\R^n$. In particular, for $T\in \D'\mathcal{T}^r_s(M)$ we have
\begin{equation}\label{eq:star_M_convergence}
T\star_M \rho_\eps \to T \ \text{ in } \ \D'\mathcal{T}^r_s(M) \quad (\eps\to 0),
\end{equation}
and indeed this convergence is even in $C^k_{\mathrm{loc}}$ or $W^{k,p}_{\mathrm{loc}}$ if $T$ is contained in these spaces (\cite[Prop.\ 3.5]{G20}).
Note that since $\rho\ge 0$, $\star_M$ preserves non-negativity: 
\begin{equation}\label{eq:star_M_preserves_positivity}
T\in \D'(M), \ T\ge 0  \ \Rightarrow \ T\star_M \rho_\eps \ge 0 \ \text{ in }  \ C^\infty(M).
\end{equation}

If $M$ is compact and $g$ is a Riemannian metric of regularity at least $C^0$, then since $g_\eps:=g\star_M \rho_\eps \to g$ uniformly on $M$,  $g_\eps$
is a smooth Riemannian metric on $M$ for $\eps$ sufficiently small (which we shall always tacitly assume below).

\section{Distributional curvature bounds via regularisation}\label{sec:dist_bounds_via_reg}
Let $M$ be a compact manifold equipped with a distributional Riemannian metric $g$. As in Section \ref{sec:dist_curv_bounds} we fix a non-negative mollifier
$\rho\in \D(B_1(0))$ with $\int \rho = 1$ and set $g_\eps := g\star_M \rho_\eps$. To begin with, we analyse the relationship between distributional
Ricci bounds for $g$ and classical (pointwise) Ricci bounds for the approximating smooth Riemannian metrics $g_\eps$. The key to this analysis are certain
versions of Friedrichs' Lemma, which provide improved convergence properties of commutators between differentiation and convolution operators.
These turned out to be essential for generalising classical singularity theorems in Lorentzian geometry to metrics of lower regularity 
(\cite{hawkingc11,penrosec11,GGKS,G20,KOSS22}). The versions we shall rely on here are the following (see \cite[Lem.\ 4.8, 4.9]{G20}):
\begin{Lemma}\label{lem:Friedrichs}\ 
\begin{itemize}
\item[(i)] Let $a\in C^1(\R^n)$, $f\in C^0(\R^n)$. Then $(a*\rho_\eps)(f*\rho_\eps) - (af)*\rho_\eps \to 0$ in $C^1(K)$ 
for any compact set $K\subseteq \R^n$.
\item[(ii)] Let $a, a_\eps \in C^1(\R^n)$ such that $a_\eps \to a$ in $C^1$ and such that for each $K$ compact in $\R^n$ there 
exists some $c_K$ such that $\|a-a_\eps\|_{\infty,K} \le c_K\eps$. Then for any $f\in C^0(\R^n)$ we have
$a_\eps(f*\rho_\eps) - (af)*\rho_\eps \to 0$ in $C^1(K)$ for any compact set $K\sse \R^n$.
\end{itemize}
\end{Lemma}
Based on this, we can collect the following commutator properties. Note that the convergence claimed in points (i)--(iii) of Proposition \ref{prop:starM_convolution_Ricci} is, in each case, one order of differentiation better than one would expect from the 
limiting properties of the individual terms. 
\begin{Proposition}\label{prop:starM_convolution_Ricci} 
Let $g$ be a $C^1$-Riemannian metric on $M$, $g_\eps:= g\star_M \rho_\eps$, and let $X, Y\in \X(M)$. Then
\begin{itemize}
\item[(i)] 
$\Ric(g\star_M \rho_\eps) - \Ric(g)\star_M \rho_\eps \to 0$ in $C^0(M,T^0_2M)$ as $\eps\to 0$.
\item[(ii)] 
$\Ric_g(X,Y)\star_M \rho_\eps - \Ric_{g_\eps}(X,Y) \to 0$ in $C^0(M)$ as $\eps\to 0$.
\item[(iii)] 
$g(X,Y)\star_M \rho_\eps - g_\eps(X,Y) \to 0$ in $C^2(M)$ as $\eps\to 0$.
\end{itemize}
\end{Proposition}
\begin{proof}
(i) This is shown in (the proof of) \cite[Lem.\ 4.6]{G20}, so we only outline the argument briefly.
Let $(\vphi, U)$ be any local chart on $M$. Then writing out the local expressions for 
$\Ric(g\star_M \rho_\eps)$ and $\Ric(g)\star_M \rho_\eps$, it
follows that it suffices to show that, setting 
$h:= \vphi_*g$, $h_\eps:= \vphi_*g_\eps$:
\[
h_\eps^{ij} ((\zeta \partial_kh_{lm}) *\rho_\eps) - (h^{ij} \zeta \partial_k h_{lm} )*\rho_\eps \to 0 \qquad (\eps\to 0)
\]
in $C^1$ for any compactly supported smooth function $\zeta$ and all $i,j,k,l,m$. 
This in turn follows from Lemma \ref{lem:Friedrichs} (ii), together with 
\cite[(4.4)]{G20}, which shows that $g_\eps^{-1}$ converges to $g^{-1}$ at least at a linear rate in $\eps$. 

(ii) Writing $R_\eps$ for a component function of a chart representation of $\Ric(g\star_M \rho_\eps)$ and $R$
for the corresponding one of $\Ric(g)$, the claim reduces to showing that, for any smooth function $h$ we
have that $(R\cdot h)*\rho_\eps - R_\eps\cdot h   \to 0$ in $C^0$. Due to (i), this in turn will follow if we can 
show that 
\begin{equation}\label{eq:RicXX_local}
(R\cdot h)*\rho_\eps - (R*\rho_\eps)\cdot h \to 0
\end{equation}
in $C^0$. Here, $R$ is a distribution of order one,
hence can locally be written as a derivative of a continuous function, which we schematically write as $R=\partial f$,
with $f$ continuous. By Lemma \ref{lem:Friedrichs} (ii), $(f*\rho_\eps)\cdot h - (fh)*\rho_\eps \to 0$ in $C^1$, so
\[
(\partial f*\rho_\eps)\cdot h + (f*\rho_\eps)\partial h - (\partial f \cdot h)*\rho_\eps - (f\partial h)*\rho_\eps \to 0
\]
in $C^0$. Here, $(f*\rho_\eps)\partial h - (f\partial h)*\rho_\eps \to 0$ in $C^0$, and the remaining terms give the claim
\eqref{eq:RicXX_local}.

(iii) Via chart representations, and schematically writing $g$ also for the local components of the metric $g$, as well
as $\partial g$ for any first order derivative,
the claim reduces to showing that, for $h\in \cinfty$, we have $(\partial g\cdot h)*\rho_\eps - (\partial g*\rho_\eps)\cdot h \to 0$ in $C^1$.
Setting $f:=\partial g\in C^0$,
\[
(f\cdot h)*\rho_\eps - (f*\rho_\eps)\cdot h = (f\cdot h)*\rho_\eps - (f*\rho_\eps)(h*\rho_\eps) + (f*\rho_\eps)(h*\rho_\eps) - (f*\rho_\eps)\cdot h.
\]
By \cite[(3.7)]{G20}, $\|h-h*\rho_\eps\|_{\infty,K} \le c_K\eps$ for any compact set $K$, and the same is true when replacing $h$ by $\partial h$ here. 
From this, convergence of the difference of the first two terms to $0$ in $C^1$ is immediate from Lemma \ref{lem:Friedrichs} (ii).
Finally, convergence to $0$ in $C^1$ of the remaining difference follows by taking into account \cite[Lem.\ 4.7]{G20}.
\end{proof}
\begin{Theorem}\label{th:Ricci_bounds_approx}
Let $K\in \R$. Then for any $C^1$-Riemannian metric $g$ on a compact manifold $M$, the following are equivalent:
\begin{itemize}
\item[(i)] $\Ric_g \ge K$ (resp.\ $\Ric_g \le K$) in the sense of distributions (see \eqref{eq:dist_Ric_bound}).
\item[(ii)] For each $\delta>0$ there exists some $\eps_0>0$ such that $\Ric_{g_\eps} \ge K-\delta$ 
(resp.\ $\Ric_{g_\eps} \le K+\delta$) for all $\eps\in (0,\eps_0)$.
\end{itemize}
\end{Theorem}
\begin{proof} It will suffice to prove the case of lower bounds.

(ii)$\Rightarrow$(i): Fix $X\in \X(M)$. Setting
$\delta:=\frac{1}{k}$, by (ii) and Proposition \ref{prop:starM_convolution_Ricci} (iii) we can pick $\eps_k \searrow 0$ such that 
(with $\Ric_{\eps}:=\Ric_{g_{\eps}}$) 
$\Ric_{\eps_k}(X,X) \ge \big(K-\frac{1}{k}\big)g(X,X)\star_M \rho_{\eps_k}$.  Then combining Proposition \ref{prop:starM_convolution_Ricci} (ii) with the fact
that $\Ric_g(X,X)\star_M \rho_\eps \to \Ric_g(X,X)$ in $\D'(M)$, it follows that $\int_M \Ric_{\eps_k}(X,X) \om \to \langle \Ric_g(X,X),\om\rangle$
for any $\om\in \Gamma_c(M,\Vol(M))$. If $\om\ge 0$, then 
\[
\int_M \Ric_{\eps_k}(X,X)\cdot \om \ge \Big(K - \frac{1}{k}\Big) \int_M (g(X,X)\star_M \rho_{\eps_k}) \cdot \omega.
\]
Letting $k\to \infty$ shows that $\lara{\Ric_{g}(X,X) - Kg(X,X)}{\omega} \ge 0$, as claimed.

(i)$\Rightarrow$(ii): We may without loss of generality assume that $M$ possesses a global frame
$X_1,\dots,X_n$ $\in \X(M)$ (otherwise cover $M$ by finitely many chart domains and argue separately). The Euclidean metric with respect to this frame defines a smooth Riemannian metric $h$ on $M$ and by compactness of $M$ there exist constants $c, C > 0$ with 
$cg(v,v) \le h(v,v) \le Cg(v,v)$ for any $v\in TM$.
Due to Proposition \ref{prop:starM_convolution_Ricci} (ii), given $\delta>0$ there exists some $\eps_1>0$ such that
for $\eps\in (0,\eps_1)$ we have
\begin{equation}\label{eq:Ricci_frame}
\Big|\sum_{i,j=1}^n \lambda_i \lambda_j (\Ric_{g_\eps}(X_i,X_j) - \Ric_g(X_i,X_j)\star_M \rho_\eps)|_x\Big| < \frac{\delta}{3C},
\end{equation}
for all $x\in M$ and all $(\lambda_1,\dots,\lambda_n)\in \R^n$ with $\sum_i \lambda_i^2 = 1$. Since the $\lambda_i$ are
independent of $x$, by definition of $\star_M$ (see \eqref{eq:M-convolution}) we may interchange $\sum \lambda_i$ and
$\star_M$ here. Thus, setting $V:=\sum_i \lambda_i X_i$, we have $h(V,V)=1$, and \eqref{eq:Ricci_frame} means that 
\begin{equation}\label{eq:Ricci_frame2}
|\Ric_{g_\eps}(V,V) - \Ric_g(V,V)\star_M \rho_\eps)|_x| < \frac{\delta}{3C},
\end{equation}
for any $x\in M$ and any $(\lambda_1,\dots,\lambda_n)\in \R^n$ as above. Analogously, it follows from Proposition \ref{prop:starM_convolution_Ricci} (iii) that there exists some $0<\eps_0\le \eps_1$ such that for any $\eps\in (0,\eps_0)$ and for any such choice of
vector field $V$ we have 
\[
|K||g(V,V)\star_M \rho_\eps-g_\eps(V,V)| < \frac{\delta}{3C}
\]
uniformly on $M$. By assumption, $\Ric_{g}(V,V) - Kg(V,V) \ge 0$ in $\D'(M)$, hence \eqref{eq:star_M_preserves_positivity} implies
that $\Ric_g(V,V)\star_M \rho_\eps - K g(V,V)\star_M \rho_\eps \ge 0$. Consequently, 
\begin{equation}\label{eq:Ricci_delta}
\Ric_{g_\eps}(V,V) - Kg_\eps(V,V) > -\frac{2\delta}{3C}  
= -\frac{2\delta}{3C} h(V,V) \ge -\frac{2\delta}{3} g(V,V).
\end{equation}
This same inequality then holds for each individual $h$-unit vector $V$ in any $T_xM$. Finally, since $g_\eps \to g$ uniformly on the $h$-unit tangent bundle,
by making $\eps_0$ smaller once more we may replace $(-2\delta/3) g(V,V)$ by $-\delta g_\eps(V,V)$ on the right hand side of \eqref{eq:Ricci_delta}, thereby
concluding the proof.
\end{proof}
\begin{remark}\label{rem:general_curvature_bounds}
The arguments used to prove Proposition \ref{prop:starM_convolution_Ricci} and Theorem \ref{th:Ricci_bounds_approx} in fact do not depend on the particular form of the Ricci tensor.
Analogous characterizations of distributional curvature bounds via regularisations therefore also hold for other curvature 
quantities, in particular for sectional curvature bounds.
\end{remark}
\section{From distributional to synthetic lower Ricci bounds}\label{sec:dist_to_synth}
In this section we show that if $M$ is a compact manifold with a $C^1$-Riemannian metric that has $K\in \R$ as a lower distributional Ricci
curvature bound, then the associated metric measure space satisfies the corresponding bound on its $\infty$-Ricci curvature
in the sense of \cite{LV09}.
Let us first recall the 
basic notions and definitions, following \cite{Vil09,LV09}. For any Polish space $(X,d)$ and probability measures $\mu, \nu\in P(X)$
denote by $W_2(\mu,\nu)$ the Wasserstein distance of order $2$ between $\mu$ and $\nu$ (cf.\ \cite[Def.\ 6.1]{Vil09}), i.e., 
\[
W_2(\mu, \nu) := \left( \inf_{\pi \in \Pi(\mu,\nu)}\int_X d(x,y)^2\, d\pi(x,y) \right)^{\frac{1}{2}}.
\]
The space of probability measures $\mu$ such that $\int_X d(x_0,x)^2\, d\mu(x) < \infty$ for some (hence any) $x_0\in X$, equipped with 
the metric $W_2$, is called the Wasserstein space of order $2$, and is denoted by $P_2(X)$. We will henceforth always assume that $X$
is compact. In the case we are mainly interested in, $X=M$ will be a compact Riemannian manifold. We then write $\pac(X)$ for the 
subspace of $P_2(X)$ of those measures that are absolutely continuous with respect to the Riemannian volume density $d\vol_g$.

Given a continuous convex function $U:[0,\infty) \to \R$ with $U(0)=0$ and $\nu\in P(X)$, define $U_\nu: P_2(X) \to \R\cup \{\infty\}$ by
\[
U_\nu(\mu) := \int_X U(\rho(x))\,d\nu(x) + U'(\infty)\mu_s(X).
\]
Here, $\mu=\rho\nu + \mu_s$ is the Lebesgue decomposition of $\mu$ with respect to $\nu$ into the absolutely continuous part $\rho\nu$ and
the singular part $\mu_s$, and $U'(\infty) :=  \lim_{r\to \infty} U(r)/r$. The space of all functions $U$ as above such that additionally
the function $\psi(\lambda):=e^{\lambda}U(e^{-\lambda})$ is convex on $\R$ is denoted by $\mathcal{DC}_\infty$. The only example of
an element of $\mathcal{DC}_\infty$ that we shall make use of is $U_\infty(r):= r\log r$. Using these notions, we now can give the
following definition (\cite[Def.\ 0.7]{LV09}):
\begin{Definition}\label{def:Ric_inf} 
A compact measured length space $(X,d,\nu)$ (with $\nu$ a probability measure) has $\infty$-Ricci curvature bounded below by $K\in \R$ if for 
all $\mu_0,\mu_1\in P_2(X)$ with $\supp(\mu_0)\sse \supp(\nu)$ and $\supp(\mu_1)\sse \supp(\nu)$ there exists a Wasserstein geodesic
$\{\mu_t\}_{t\in [0,1]}$ from $\mu_0$ to $\mu_1$ such that for all $U\in \mathcal{DC}_\infty$ and all $t\in [0,1]$ we have
\begin{equation}\label{eq:weak_displacement_convexity}
U_\nu(\mu_t) \le tU_\nu(\mu_1) + (1-t)U_\nu(\mu_0) - \frac{1}{2} \lambda_K(U) t(1-t)W_2(\mu_0,\mu_1)^2.
\end{equation}
Here, the function $\lambda_K$ is defined as follows: Let $p(r):=rU_+'(r) - U(r)$, $p(0)=0$, then $\lambda_K(U):=\inf_{r>0} K\frac{p(r)}{r}$
(see \cite[Def.\ 5.13]{LV09}).
\end{Definition}
This property is called \emph{weak displacement convexity}. We note that in the case of $U_\infty$ we obtain $\lambda_K(U_\infty) = K$.

For $(M,g)$ a compact connected Riemannian manifold with volume form $d\vol_g$, letting 
\begin{equation}\label{eq:nu_def}
\nu_g:= \frac{d\vol_g}{\vol_g(M)},
\end{equation}
and denoting 
by $d_g$ the Riemannian distance induced by $g$, \cite[Th.\ 0.12, Th.\ 7.3]{LV09} imply:
\begin{Theorem}\label{th:Ric_equivalence_smooth}
Let $M$ be a compact connected manifold with a Riemannian metric $g$ of regularity $C^2$. Then the measured length space $(M,d_g,\nu_g)$ has $\infty$-Ricci
curvature bounded below by $K\in \R$ if and only if $\Ric_g \ge Kg$. 
\end{Theorem}
Let now $g$ be a Riemannian metric of regularity $C^1$ on the compact manifold $M$, and let $g_\eps:= g\star_M \rho_\eps$ be as in \eqref{eq:M-convolution}.
We first study the convergence of the measured length spaces $(M,d_{g_\eps},\vol_{g_\eps})$ towards $(M,d_g,d\vol_g)$:
\begin{Proposition}\label{prop:measured_GH_geps}
Let $M$ be a compact connected manifold with a $C^0$-Riemannian metric $g$, and let $g_\eps = g\star_M \rho_\eps$. Then 
$(M,d_{g_\eps},\vol_{g_\eps}) \to (M,d_g,d\vol_g)$ in the measured Gromov-Hausdorff sense.
\end{Proposition}
\begin{proof} Since $g_\eps \to g$ uniformly on $M$, for any $\delta>0$ there exists some $\eps_0>0$ such that
for any $\eps\in (0,\eps_0)$ we have
\begin{equation}\label{eq:distance_relations}
(1-\delta)g(v,v) \le g_\eps(v,v) \le (1+\delta)g(v,v) \qquad \forall v\in T_pM \ \forall p\in M.
\end{equation}
Moreover, also by uniform convergence, $\vol_{g_\eps}(M) \to \vol_{g}(M)$. The claim therefore follows from
\cite[Th.\ 1.2]{AS20}.
\end{proof}
Using this result, we can now show:
\begin{Theorem}\label{th:dist_to_synth}
Let $M$ be a compact connected manifold with a $C^1$-Riemannian metric $g$ that satisfies $\Ric_g \ge Kg$ in the distributional sense
(see \eqref{eq:dist_Ric_bound}). Then $(X,d_g,\nu_g)$ has $\infty$-Ricci curvature bounded below by $K$.
\end{Theorem}
\begin{proof} Fix $\delta>0$. By Theorem \ref{th:Ricci_bounds_approx}, there exists
some $\eps_0>0$ such that, for any $0<\eps<\eps_0$, the smooth metric $g_\eps$ satisfies $\Ric_{g_\eps} \ge (K-\delta) g_\eps$. 
Furthermore, Proposition \ref{prop:measured_GH_geps} shows that $(M,d_{g_\eps},\nu_{g_\eps}) \to (M,d_g,\nu_g)$ in the 
measured Gromov-Hausdorff sense. Also, due to Theorem \ref{th:Ric_equivalence_smooth}, each $(M,d_{g_\eps},\nu_{g_\eps})$
has $\infty$-Ricci curvature bounded below by $K-\delta$. We may now employ the stability of weak displacement convexity
\cite[Th.\ 4.15]{LV09} to conclude that for any $U\in \mathcal{DC}_\infty$, $U_{\nu_g}$ is weakly $\lambda_{K-\delta}$-displacement
convex. Since this holds for any $\delta>0$, the claim follows.
\end{proof}
\section{From synthetic to distributional lower Ricci bounds}\label{sec:synth_to_dist}
We now turn to the converse of the implication considered previously. When trying to infer distributional Ricci bounds 
from synthetic ones, one faces significantly greater technical difficulties. Indeed one can observe that when deriving classical
bounds for smooth (at least $C^2$-) Riemannian metrics from synthetic ones, the standard proofs (e.g., \cite[Th.\ 7.3]{LV09}, \cite[Th.\ 1.1]{vRS})
make use of analytic tools that cease to be available in regularities strictly below $C^2$, e.g., Jacobi fields, or estimates of curvature
quantities along geodesics, both of which would require the evaluation of curvature terms (which are, at best, only defined almost everywhere)
along curves, hence along null sets. Our strategy will again be to work with regularised smooth metrics $g_\eps = g\star_M \rho_\eps$, and 
we will assume the metric $g$ itself to be of class $C^{1,1}$ (continuously differentiable with Lipschitz continuous first derivatives).
This regularity class has turned out to be of considerable interest in applications to General Relativity, since it still guarantees,
on the one hand, unique solvability of the geodesic equations, and local boundedness of all curvature quantities 
(cf.\ \cite{hawkingc11,penrosec11,GGKS,Min15}). From the technical point of view, such metrics still provide enough control over curvature 
quantities of the approximating metrics $g_\eps$ to suitably adapt arguments from the smooth setting.

Suppose that $g$ is a $C^{1,1}$-Riemannian metric on a compact connected manifold $M$ such that $\Ric_g$ does not have $K$ as a lower distributional bound. By Theorem \ref{th:Ricci_bounds_approx} this
means that for some $\delta>0$ there exists a sequence $\eps_k \searrow 0$ and vectors $v_k \in T_{x_k}M$, 
such that 
\begin{equation}\label{eq:indirect_assumption}
\Ric_{g_{\eps_k}}(v_k,v_k)< (K-\delta)g_{\eps_k}(v_k,v_k).
\end{equation}
Henceforth we write $g_k$ for $g_{\eps_k}$. By compactness we may suppose that $x_k \to x_0$, $v_k\to v$,
$v\in T_{x_0}M$. For later use we note that \eqref{eq:indirect_assumption} remains true if we multiply $v_k$ by
a constant, so that
by re-scaling we may assume the norms of $v_k$, $v$ as small as
we wish. Our aim in this section 
is to derive a contradiction to $g$ possessing $K$ as a lower $\infty$-Ricci bound by constructing a Wasserstein geodesic 
along which displacement convexity \eqref{eq:weak_displacement_convexity} fails for the function $U_\infty(r) = r\log r$. 
To this end, we will follow the basic structure of the proof of \cite[Th.\ 7.3]{LV09}.

To begin with, for any $k$ we choose a smooth map $\phi_k:M\to \R$ such that
\begin{equation}\label{eq:phik_assumptions}
\nabla^{g_k}\phi_k(x_k)=-v_k, \quad{\text{ and }} \quad  \mathrm{Hess}^{g_k}(\phi_k)(x_k)=0,
\end{equation}
as follows:  
Working in local coordinates centered at $x_0$, we define 
\begin{equation*}\label{eq:phikdef}
\tilde \phi_k(x) := (g_k(x_k))_{il}v_k^i (x_k^l - x^l) - \frac{1}{2} {}^{g_k}\Gamma^l_{ij}(x_k) (g_k(x_k))_{rl} v^r_k (x^i-x_k^i)(x^j-x_k^j)
\end{equation*}
in a neighbourhood of $x_0=0$ and extend this function to all of $M$ as $\phi_k:=\zeta \cdot \tilde \phi_k$, where 
$\zeta$ is a plateau function that equals $1$ in a neighbourhood of $x_0$ (to be further specified below). Here, ${}^{g_k}\Gamma^l_{ij}$ denote the Christoffel symbols of $g_k$
in the local coordinates. Clearly, $\phi_k \to \phi = \zeta \tilde \phi$ in $C^\infty(M)$, where
\begin{equation*}
\tilde \phi(x) = (g(x_0))_{il}v^i (x_0^l - x^l) - \frac{1}{2} {}^{g}\Gamma^l_{ij}(x_0) (g(x_0))_{rl} v^r (x^i-x_0^i)(x^j-x_0^j),
\end{equation*}
and we have $\nabla^{g}\phi(x_0)=-v$ and $\mathrm{Hess}^{g}(\phi)(x_0)=0$.

As in Section \ref{sec:fundamentals_opt_transp}, let $c(x,y):=d(x,y)^2/2$, with $d$ the Riemannian distance induced by $g$
and analogously $c_k(x,y):=d_k(x,y)^2/2$ for the metric $g_k$. We next want to show that there exists some constant $\kappa>0$
such that, if $\|\nabla^{g_k} \phi_k\|_\infty \le \kappa$ (which can be achieved by making 
$\|v_k\|_{g_k}$ uniformly small), then
$\phi_k$ is $c_k$-concave for each $k\in \N$. To this end, we employ the following result from 
\cite{glaudo}:
\begin{Theorem}\label{th:glaudo} (\cite[Th.\ 1.1]{glaudo}) Let $(M,g)$ be a compact Riemannian manifold with sectional curvature bounded from above
by $K\ge 0$. Then there exists a constant $C_*:=C_*(\mathrm{inj}(M),$ $K,\mathrm{diam}(M))>0$ such that, for any $\eps>0$, if 
$\phi\in C^2(M,\R)$ satisfies 
\[
\| \nabla \phi \|_\infty \le \min\left( \frac{\eps}{3K \mathrm{diam}(M)},C_* \right) \qquad \text{ and } \qquad \mathrm{Hess}(\phi) \le (1-\eps)g,
\]
then $\phi$ is $c$-concave.
\end{Theorem}
Thus to establish the above claim we need to see that when applying this result to $g_k$, all the quantities used in the estimates can be 
controlled uniformly in $k$. For the injectivity radii $\mathrm{inj}_{g_k}$ this follows from \cite[Th.\ 4.7]{CGT}, to the effect that a uniform  
bound on the Riemann tensor combined with a uniform lower bound on the volumes of distance balls of radius $1$ gives a lower bound on the 
injectivity radius (cf.\ \cite[Th.\ 3.3]{KSS14} and the discussion following it). Moreover, both the Hessians with respect to $g_k$ and the 
$g_k$-diameters of $M$ converge uniformly, so the claim follows. We henceforth assume that the $\|v_k\|_{g_k}$ are sufficently small
to guarantee $c_k$-concavity of $\phi_k$.

Let $\eta_0^{(k)} := \vol_{g_k}(V)^{-1} 1_V$ be a uniform distribution on some open neighbourhood $V$ of $x_0$ (to be specified more precisely below,
but in any case such that $\phi_k = \tilde\phi_k$ and $\phi = \tilde \phi$ on $V$), 
$\eta_0 := \vol_{g}(V)^{-1} 1_V$, $\mu_0^{(k)} := \eta_0^{(k)} d\vol_{g_k}$, and $\mu_0 := \eta_0 d\vol_{g}$. 
 Then $\mu_0^{(k)}\ll d\vol_{g_k}$ and $\mu_0\ll d\vol_{g}$ are probability measures.
Set $F_t^{(k)}(y):=\exp^{g_k}_y(-t\nabla^{g_k} \phi_k)$, $F_t(y):=\exp^{g}_y(-t \nabla^g \phi)$, 
and $F^{(k)} := F_1^{(k)}$, $F := F_1$. Since $\phi_k$ is $c_k$-concave, \cite[Th.\ 8]{McC01}, \cite[Cor.\ 5.2]{CMcCS} show that, 
for each $t\in [0,1]$, $\mu_t^{(k)} := (F_{t}^{(k)})_\#\mu_0^{(k)}$ is a $c_k$-optimal transport from $\mu_0^{(k)}$ to $\mu_t^{(k)}$. Consequently,
$t\mapsto \mu_t^{(k)}$ is a geodesic for the $2$-Wasserstein distance induced by $g_k$. Furthermore, denoting by $W^{(k)}_2$ and $W_2$ the Wasserstein distances
induced by $d_{g_k}$ and $d_g$, respectively, \eqref{eq:distance_relations} shows that by picking a subsequence we may assume without loss of generality that
for each $k$ we have $(1-1/k)W_2(\mu,\nu) \le W^{(k)}_2(\mu,\nu) \le (1+1/k)W_2(\mu,\nu)$ for all $\mu, \nu\in P_2(M)$. 
Since $P_2(M)$ is compact (\cite[Rem.\ 6.19]{Vil09}), we are precisely in the setting of the following auxilliary result:
\begin{Lemma}\label{lem:geod_in_metric} Let $(X,d)$ be a compact metric space and let $d_k$ be a sequence of metrics on $X$ such that, for each $k\in \N$,
\begin{equation}\label{eq:dkdinequality}
(1-1/k)d(x,y) \le d_k(x,y) \le (1+1/k)d(x,y)
\end{equation}
for all $x,y,\in X$. Let $\gamma_k: [0,1] \to X$ be a $d_k$-geodesic. Then there exists a subsequence $(\gamma_{k_l})$ of $(\gamma_k)$ that converges
uniformly to a $d$-geodesic $\gamma: [0,1]\to X$ with $d$-length $L_d(\gamma) = \lim_{l\to\infty} L_{d_{k_l}}(\gamma_{k_l})$.
\end{Lemma}
\begin{proof} By compactness, we may suppose that $\gamma_k(0)\to p\in X$ and $\gamma_k(1)\to q\in X$. Also, we may assume the
$\gamma_k$ to be parametrized proportional to arclength, so $d_k(\gamma_k(s),\gamma_k(t)) = d_k(\gamma_k(0),\gamma_k(1))|s-t|$.
Then by \eqref{eq:dkdinequality} we obtain
\[
d(\gamma_k(s),\gamma_k(t)) \le 2d(p,q)|s-t|
\]
for $k$ large and any $s,t\in [0,1]$, showing that $(\gamma_k)_k$ is uniformly equicontinuous, and pointwise bounded as $X$ is compact. 
The existence of a uniformly convergent subsequence therefore follows from the Arzela-Ascoli theorem. For simplicity, we denote
this subsequence again by $(\gamma_k)_k$.

Next, fix $\eps>0$ and denote by $\sigma = \{t_0=0 < \dots < t_{n_\sigma}=1\}$ any subdivision of the interval $[0,1]$, then
there exists some $k_0 = k_0(\eps,\sigma)$ such that for $k\ge k_0$ we have
\begin{align*}
\sum_{i=1}^{n_\sigma} d(\gamma(t_{i-1}),\gamma(t_i)) \le \sum_{i=1}^{n_\sigma} d_k(\gamma_k(t_{i-1}),\gamma_k(t_i)) + \eps
= L_{d_k}(\gamma_k) + \eps \le d(\gamma(0),\gamma(1)) + 2\eps.
\end{align*}
Consequently,
\[
d(\gamma(0),\gamma(1)) \le \sum_{i=1}^{n_\sigma} d(\gamma(t_{i-1}),\gamma(t_i)) \le  d(\gamma(0),\gamma(1)) + 2\eps.
\]
Letting $\eps\to0$ and taking the supremum over all $\sigma$, we obtain $d(\gamma(0),\gamma(1)) = L_d(\gamma)$.
Finally, $L_{d_k}(\gamma_k) = d_k(\gamma_k(0),\gamma_k(1)) \to d(\gamma(0),\gamma(1)) = L_d(\gamma)$.
\end{proof}
Thus, up to picking another subsequence, we conclude that $\mu_t^{(k)}$ converges to a $W_2$-geodesic $\chi_t$.
Since $\sqrt{\det{g_k}} \to \sqrt{\det{g}}$ uniformly on $M$ and $\vol_{g_k}(V)\to \vol_g(V)$, by \cite[Th.\ 6.9]{Vil09}  $\mu_0^{(k)} \to \mu_0$ in $W_2$. From the above it therefore follows that, in particular, $\chi_1$ is an optimal transport of $\mu_0$, hence by Theorem \ref{th:opt_maps} is of the 
form $\chi_1 = T_\#\mu_0$, where $T(y) = \exp_y( -\nabla \psi(y))$ for some $c$-concave function $\psi: M\to \R$. 
Note that henceforth we will often simply write $\exp$ instead of $\exp^g$, $\nabla$ instead of $\nabla^g$, etc.

We wish to relate this function $\psi$ to the limiting function $\phi$ from above. To this end, we 
will make use of results on the strong differentiability of the exponential map of a $C^{1,1}$-metric
over the zero-section of $TM$ (\cite[Th.\ 3]{Min15}). Denote by $E$ the map, defined on a neighbourhood 
of $M\times \{0\}$ in $TM$, which maps $(y,w)$ to $(y,\exp_y(w))$. 
By \cite[Th.\ 3]{Min15}, $E$ is strongly differentiable in any point in $M\times \{0\}$. In particular,
choosing, as above, local coordinates such that $x_0=0$, the strong differential of $E$ in $(0,0)$
is given by
\[
L:= \begin{pmatrix}
    I & 0 \\ I & I
    \end{pmatrix},
\]
Indeed, the estimate given below \cite[(39)]{Min15} shows that, for $\max(\|y\|,\|z\|,\|v\|,\|w\|) < \delta$,
$\delta$ sufficently small, we have
\begin{equation}\label{eq:E_strongly_differentiable}
\|E(y,v) - E(z,w) - L\cdot (y-z,v-w)^\intercal \| \le \max(\|y-z\|,\|v-w\|) \cdot O(h(\delta)).
\end{equation}
Here, $h(\delta)\to 0$ as $\delta\to 0$. We pick $\delta>0$ such that $h(\delta)<1/4$. 
Now $F_t(y) = \mathrm{pr}_2\circ E(y,t\nabla \phi(y))$, and rescaling $v$ suitably
we can assume that $\|\nabla \phi(y)\|, \|\nabla\phi(z)\| \le \delta$ and $\|D^2\phi\|_\infty\le 1/2$, so that
$\|\nabla \phi(y) - \nabla \phi(z)\| \le \frac{1}{2} \|y-z\|$ for $y, z$ in a $\delta$-ball around $0$.
Then \eqref{eq:E_strongly_differentiable} gives
\begin{equation*}
\|F_t(y) - F_t(z) - (y-z) - t(\nabla\phi(y) - \nabla\phi(z)) \| \le h(\delta) \|y-z\|.
\end{equation*}
Thus
\begin{equation*}
\|y-z - t(\nabla\phi(y) - \nabla\phi(z))\| - h(\delta) \|y-z\| \le \|F_t(y) - F_t(z) \|
\end{equation*}
for all $t\in [0,1]$. By the choices made above this entails
\begin{equation*}
\|F_t(y) - F_t(z) \| = \|\exp_y(-t\nabla\phi(y)) - \exp_z(-t\nabla\phi(z))\| \ge \frac{1}{4}\|y-z\|.
\end{equation*}
It follows that there exists a neighbourhood of $x_0$ on which, for any $t\in [0,1]$, $F_t$ is a bi-Lipschitz homeomorphism.
We note that the same calculation can be carried out for each $F^{(k)}_t$, and in fact both $\delta$ and $h(\delta)$
can be chosen uniformly for all $F^{(k)}_t$ and $F_t$ since they only depend on the bounds on the Christoffel symbols
and the Lipschitz constants for the exponential maps (cf.\ \cite[(27),(28),(33)]{Min15}), all of which are uniformly 
controlled since $g\in C^{1,1}$.

From here, the standard change-of-variables formula for bi-Lipschitz maps implies that given a measure $\xi_0 d\vol_g$, its
push-forward under $F_t$ possesses a density with respect to $d\vol_g$, namely the map
\begin{equation}\label{eq:Ft_density}
x\mapsto \xi_0(F_t^{-1}(x)) \left.\frac{1}{\det DF_t(y)}\right|_{y=F_t^{-1}(x)}.
\end{equation}
An analogous equation holds for each $F_t^{(k)}$. We also note that due to the uniform bound we have on the Lipschitz constants
of the $g_k$-exponential maps it follows directly from the proofs of the main theorem in \cite{leach} and that of \cite[Th.\ 3]{Min15} 
that the inverses of $F_t$ and $F_t^{(k)}$ are defined on a neighbourhood of $F_t(x_0)$ resp.\ $F_t^{(k)}(x_0)$ whose size is independent of 
$t\in [0,1]$ and of $k$.  
Rescaling the $v_k$ we may therefore assume that all inverses of the $F_t^{(k)}$ are defined on a neighbourhood $\tilde V$ of $x_0=F_0(x_0)
=F_0^{(k)}(x_0)$, and
that they converge locally uniformly to $F_t^{-1}$ on $\tilde V$ (cf.\ \cite{BDF91}). Finally, we pick a neighbourhood $V$ of $x_0$ whose
image under all $F_t^{(k)}$ and $F_t$ lies in $\tilde V$.

Taking this $V$ in the definition of $\mu_t^{(k)}$ above, we want to show that the densities of the $\mu_t^{(k)}$ 
converge to the density of $\mu_t:=(F_t)_\sharp\mu_0$. This will 
require us to uniformly control the functions $\det DF^{(k)}_t$ from below, a property that will also be needed in 
a later stage of the proof. We will derive such an estimate by Riccati comparison of the Jacobian differentials of the 
exponential maps of the $g_k$ (cf., e.g., \cite[Sec.\ 1.5]{daiwei} or \cite[Sec.\ 7]{LV09}).

Fix some $k\in \N$ and, given $y\in M$, consider the $g_k$-geodesic $\gamma(t):= F_t^{(k)}(y) = \exp^{g_k}_y(-t \nabla^{g_k}\phi_k(y))$. Also, let $e_i$ be
a $g_k$-orthonormal basis at $y$, parallely transported along $\gamma$.
Let $J_i$ be the Jacobi-field 
\begin{equation}\label{eq:Jacobi-field}
J_i(t):= DF^{(k)}_t(e_i)
\end{equation}
along $\gamma$. Then $J_i(0)=e_i$ and $J_i'(0)= \mathrm{Hess}^{g_k}(\phi_k)_y(e_i)$
Setting $J_{ij}:= \lara{J_i}{e_j}_{g_k}$, the matrix $J:=(J_{ij})$ satisfies the initial value problem
\begin{equation*}
J''(t) + K(t) J(t) = 0, \quad J(0) = I_n, \quad J'(0) = \mathrm{Hess}^{g_k}(\phi_k)_y.
\end{equation*}
Here, $K_{ij}(t) = \lara{R^{g_k}(e_i(t),\dot\gamma(t))\dot \gamma(t)}{e_j(t)}_{g_k(\gamma(t))}$. Let $U(t) := J'(t)\cdot J^{-1}(t)$, 
$\cJ(t) := \det J(t) = \det DF^{(k)}_t(y)$, and
\begin{equation*}
h(t) := \log \cJ(t) = \log \det DF^{(k)}_t(y).
\end{equation*}
Then
\begin{equation*}
\mathrm{tr} (U) = \frac{d}{dt}(\log \cJ) = \dot h,
\end{equation*}
and $U$ satisfies the matrix Riccati equation (cf.\ \cite[(1.5.2)]{daiwei})
\begin{equation*}
\dot U + U^2 + K = 0
\end{equation*}
with initial condition $U(0) = \mathrm{Hess}^{g_k}(\phi_k)_y$. Since $g\in C^{1,1}$, $K$ is bounded, independently of $k$ (cf.\ Proposition
\ref{prop:starM_convolution_Ricci} and Remark \ref{rem:general_curvature_bounds}), so there exists some constant $H>0$ such that
\begin{equation*}
-H\cdot I_n \le K \le H\cdot I_n
\end{equation*}
in the sense of symmetric bilinear forms. Consider now the comparison equations
\begin{align*}
\dot U_{-H} + U_{-H}^2 - H\cdot I_n &= 0\\
\dot U_H + U_H^2 + H\cdot I_n &= 0
\end{align*}
with identical initial condition $U_H(0) = U_{-H}(0) = U(0)$. The main theorem of \cite{EH} then implies that
\begin{equation*}
U_H(t) \le U(t) \le U_{-H}(t)
\end{equation*}
on any common existence interval $[0,\bar t]$ of $U_H$, $U_{-H}$.

To solve the comparison equations it suffices to take a $g_k$-orthogonal matrix $T$ that diagonalizes the initial condition.
Then $\tilde U_H := T^{-1}U_HT$ solves the same equation, but with diagonal initial conditions, hence the system decouples and
can be solved explicitly, namely $\tilde U_H(t) = \mathrm{diag}(s_H^{1}(t),\dots,s_H^n(t))$, 
$\tilde U_{-H}(t) = \mathrm{diag}(s_{-H}^{1}(t),\dots,s_{-H}^n(t))$ for smooth functions 
$s_{H}^{i}$, $s_{-H}^{i}$ on $[0,\bar t]$, given by 
\begin{align*}
s^i_H&=-\sqrt{H}\tan\left(t\sqrt{H}-\arctan(s^i_H(0)/\sqrt{H})\right), \\
s^i_{-H}&=\sqrt{H}\tanh\left(t\sqrt{H}+\tanh^{-1}(s^i_{-H}(0)/\sqrt{H})\right).
\end{align*}
Transforming back it follows that
\[
\min_{1\le i\le n} s_{H}^{i}(t) I_n \le U(t) \le \max_{1\le i\le n} s_{-H}^{i}(t) I_n
\]
on $[0,\bar t]$. Consequently,
\[
n \min_{1\le i\le n} s_{H}^{i}(t)  \le \mathrm{tr}(U(t)) = \dot h(t) \le n \max_{1\le i\le n} s_{-H}^{i}(t)
\]
on $[0,\bar t]$, which in turn implies a uniform bound on $h(t)$, independently of $k$.

We conclude that, in particular, 
\begin{equation}\label{eq:logdetFk_bounded}
\log(\det DF^{(k)}_t(y)) \ \text{is bounded, uniformly in } \ k\in \N, \ y\in M, \ t\in [0,\bar t].
\end{equation}
Hence also $\det DF^{(k)}_t(y))$ is uniformly bounded below. From the explicit form of $F_t^{(k)}$ it follows that, rescaling the $v_k$
by a factor independent of $k$, we may assume without loss of generality that $\bar t = 1$.

We now make the following technical assumption (cf.\ Remark \ref{rem:discussion} below for a discussion): 

{\bf Assumption:} \emph{There exists a (Lebesgue-) null set $N\sse M$ 
such that, for each $y\in M\setminus N$,
\begin{equation}\label{eq:additional_assumption}
DF_t^{(k)}(y) \to DF_t(y), 
\end{equation}
uniformly for $t\in [0,1]$.}

Together with the above, \eqref{eq:Ft_density} for $F_t^{(k)}$, and the fact that $\vol_{g_k} \to \vol_g$ uniformly on $M$ imply that $\mu_t^{(k)} \to \mu_t$ weakly for each $t\in [0,1]$.
Since, on the other hand, convergence in the Wasserstein sense implies weak convergence, it follows that $\chi_t = \mu_t$
for all $t\in [0,1]$. 

To continue the argument, we now require the validity of the following two standard results also for $C^{1,1}$-Riemannian metrics:
First, the proof of \cite[Cor.\ 3.22]{AG} relies on \cite[Th.\ 3.10]{AG}, which clearly applies to the current setting, as 
well as on \cite[Rem.\ 2.35]{AG}, which corresponds to our Lemma \ref{lem:subdiff}, so remains true for $g\in C^{1,1}$.
Furthermore, \cite[Cor.\ 5.2]{CMcCS} is derived from \cite[Lem.\ 5.1]{CMcCS}, which carries over verbatim to the $C^{1,1}$-setting,
together with McCann's characterization of optimal transport, which in this
regularity takes the form of Theorem \ref{th:opt_maps}, showing that also this result holds for $g\in C^{1,1}$.

Combining these two results we conclude that with $\psi$ the 
$c$-concave function introduced after the proof of Lemma \ref{lem:geod_in_metric} 
we have $\chi_t = (H_t)_\sharp \mu_0$, where $H_t = y\mapsto \exp_y(-t\nabla \psi(y))$. 
By Theorem \ref{th:opt_maps} we 
therefore have, for each $t\in [0,1]$, that $\exp_y(-t\nabla \phi(y)) = \exp_y(-t\nabla \psi(y))$
$\mu_0$-almost everywhere and hence Lebesgue-almost everywhere on $V$. We now again use the fact that by \cite[Th.\ 3]{Min15}, 
the map $TM\to M\times M$, $v_y \mapsto \exp^g_y(v_y)$ is a bi-Lipschitz homeomorphism on a neighbourhood of the zero-section in $TM$.
For $t>0$ small, both $t\nabla \phi$ and $t\nabla \psi$ lie in this neighbourhood, hence coincide almost everywhere on $V$.
Since, in all further calculations, $\phi$
enters only via $\nabla\phi$, and only measures supported in $V$ will be considered, we may without loss of generality henceforth assume that $\phi$ itself is $c$-concave. We also note that, by what was said above, $\mu_t$ is the unique Wasserstein geodesic between 
$\mu_0$ and $\mu_1$.

Now that we know $c$-concavity of $\phi$, we consider more general initial densities $\eta_0$ than the constant one we used
above. Thus, let $\mu_0=\eta_0 d\vol_g$ be any probability measure on $V$ with density $\eta_0$ with respect to $d\vol_g$.
Then since $\phi$ is $c$-concave, $\mu_t := (F_t)_\#\mu_0$ is a $d_g$-Wasserstein geodesic, and in fact is the unique
such geodesic by what was said above. Setting  
$U \equiv U_\infty = r \log r$, by Definition \ref{def:Ric_inf} we therefore have for all $t\in [0,1]$:
\begin{equation}\label{eq:Wasserstein_Ricci}
U_\nu(\mu_t) \le tU_\nu(\mu_1) + (1-t)U_\nu(\mu_0) - \frac{1}{2} K t(1-t)W_2(\mu_0,\mu_1)^2,
\end{equation}
where $\nu$ was given in \eqref{eq:nu_def}. As in the proof of \cite[Th.\ 7.3]{LV09} (but now using the standard transformation
formula for bi-Lipschitz homeomorphisms for $F_t$), we have
\begin{equation*}
U_\nu(\mu_t) = \int_M U\Big(\vol_g(M)\cdot \frac{\eta_0(y)}{\det(DF_t)(y)} \Big) \det(DF_t)(y) \frac{d\vol_g(y)}{\vol_g(M)}.
\end{equation*}
Then setting 
\begin{equation*}
C(y,t) := - \log(\vol_g(M)) + \log(\det DF_t(y)),
\end{equation*}
and noting that also for $g\in C^{1,1}$ we have $W_2(\mu_0,\mu_1)^2 = \int_M |\nabla\phi(y)|^2 \eta_0(y)\, d\vol_g(y)$,
it follows that \eqref{eq:Wasserstein_Ricci} is equivalent to
\begin{equation}\label{eq:Wasserstein_Ricci_equiv}
\begin{split}
-\int_M \eta_0(y) & C(y,t)  d\vol_g(y) \\
&\le \int_M \eta_0(y)\Big(-tC(y,1) - (1-t)C(y,0) -\frac{1}{2}Kt(1-t) |\nabla\phi(y)|^2 \Big)\,d\vol_g(y).
\end{split}
\end{equation}
Since this inequality is invariant under non-negative scaling, $\eta_0$ can be any non-negative Borel measurable function on $V$ here.

Arguing in a coordinate chart around $x_0$ (hence setting $M=\R^n$ and $x_0=0$ for the moment), we now let $\zeta: \R^n\to [0,1]$ be smooth with compact support in 
the unit ball and $\zeta(0)=1$ and set, for $\eps \in (0,1]$, $\eta_{0\eps}(x) := \eps\cdot \zeta(x/\eps)$, and $\eta_{0\eps}^{(k)}(x) := \eps\cdot \zeta((x-x_k)/\eps)$. Then
\begin{equation}\label{eq:eta0eps}
|\eta_{0\eps}(x) - \eta_{0\eps}^{(k)}(x)| \le \|D\zeta\|_{\infty}\cdot \|x_k\| \to 0 \qquad (k\to \infty),
\end{equation}
uniformly in $x$ and $\eps$.   Also, setting $C_k(y,t) := - \log(\vol_{g_k}(M)) + \log(\det DF^{(k)}_t(y))$, \eqref{eq:logdetFk_bounded} and \eqref{eq:additional_assumption} 
imply that $|C(y,t)-C_k(y,t)|$ is uniformly bounded in $y, t, k$ and converges pointwise to zero for almost all $y$, uniformly in $t$, as $k\to \infty$.
Together with the fact that $\nabla^{g_k}\phi_k \to \nabla^g\phi$ uniformly on $M$, we conclude that
\begin{align*}
\int_M \eta_{0\eps}^{(k)}(y)  C_k(y,t)  d\vol_{g_k}(y) \to \int_M \eta_{0\eps}(y)  C(y,t)  d\vol_g(y)
\end{align*}
as well as
\begin{align*}
\int_M \eta_{0\eps}^{(k)}(y)\big(-tC_k(y,1) - (1-t)C_k(y,0) -\frac{1}{2}Kt(1-t) |\nabla^{g_k}\phi_k(y)|^2 \big)\,d\vol_{g_k}(y) \\
\to
\int_M \eta_0(y)\big(-tC(y,1) - (1-t)C(y,0) -\frac{1}{2}Kt(1-t) |\nabla\phi(y)|^2 \big)\,d\vol_g(y),
\end{align*}
as $k\to \infty$, uniformly in $t\in [0,1]$ and in $\eps\in (0,1]$. Due to \eqref{eq:Wasserstein_Ricci_equiv} we conclude that there exists
some $k_0$ such that, for all $k\ge k_0$, any $\eps\in (0,1]$ and any $t\in [0,1]$ we have
\begin{equation}\label{eq:Wasserstein_Ricci_equiv_contra}
\begin{split}
&- \int_M  \eta_{0\eps}^{(k)}(y)  C_k(y,t)  d\vol_{g_k}(y) \\
&\le \int_M \eta_{0\eps}^{(k)}(y)\Big(-tC_k(y,1) - (1-t)C_k(y,0) -\frac{1}{2}\Big(K-\frac{\delta}{2}\Big) t(1-t) |\nabla^{g_k}\phi_k(y)|^2\Big) \,d\vol_{g_k}(y).
\end{split}
\end{equation}
Letting $\eps\to 0$, the support of $\eta_{0\eps}^{(k)}$ can be concentrated arbitrarily close to $x_k$. The Fundamental Lemma of the calculus of variations yields that
\begin{equation}\label{eq:Wasserstein_Ricci_infinitesimal}
   - C_k(x_k,t) \leq -t C_k(x_k,1) - (1-t) C_k(x_k,0) - \frac{1}{2}\Big(K-\frac{\delta}{2}\Big) t(1-t) |\nabla^{g_k}\phi_k(x_k)|^2
\end{equation}
for all $t\in[0,1]$. Using that $|\nabla^{g_k}\phi_k(x_k)|^2 = g_k(v_k,v_k)$ and adding $-\frac{1}{2}\Big(K-\frac{\delta}{2}\Big) t^2 g_k(v_k,v_k)$ on both sides of \eqref{eq:Wasserstein_Ricci_infinitesimal},
we see that the function
\[
  - C_k(x_k,t) -\frac{1}{2}\Big(K-\frac{\delta}{2}\Big) t^2 g_k(v_k,v_k)
\]
is convex. In particular,
\begin{equation}\label{eq:convexity_at_x_k}
-\frac{\partial^2}{\partial t^2} C_k(x_k,0) - \Big(K-\frac{\delta}{2}\Big) g_k(v_k,v_k) \geq 0.
\end{equation}

We now note that  
\begin{equation}\label{eq:Cdoubleprime}
\frac{\partial^2}{\partial t^2} C_k(x_k,0) = - \Ric_{g_k}(v_k,v_k)
\end{equation}
To see this we follow the line of argument in \cite[Lem.\ 7.4]{LV09}. Fixing $k$ and setting $D(t):= \det^{\frac{1}{n}}(DF_t^{(k)}(x_k))$, 
we have $C(t):=C_k(x_k,t) = \log(\vol_{g_k}(M)) + n\log (D(t))$. By \cite[(7.24), (7.30)]{LV09} and \eqref{eq:phik_assumptions},
\begin{equation*}
\frac{D'(0)}{D(0)} = -\frac{1}{n} \Delta_{g_k}\phi_k(x_k) = 0,
\end{equation*}
hence 
\begin{equation}\label{eq:Cpp}
C''(0) = n \frac{D''(0)}{D(0)} - n\Big(\frac{D'(0)}{D(0)} \Big)^2 = n \frac{D''(0)}{D(0)}.
\end{equation}
On the other hand, by \cite[(7.16)]{LV09},
\begin{equation}\label{eq:Dpp}
\frac{D''(t)}{D(t)} = \frac{1}{n^2} (\mathrm{Tr}(R))^2 - \frac{1}{n}\mathrm{Tr}(R^2) - \frac{1}{n}\Ric_{g_k}((F_t^{(k)})'(x_k),(F_t^{(k)})'(x_k)).
\end{equation}
Here, taking $J_i$ as in \eqref{eq:Jacobi-field} (with $y=x_k$), $R(t)$ is the matrix defined by $J_i'(t) = \sum_j R(t)^j_i J_j(t)$. Thus
$R(0)$ is the matrix of the $g_k$-Hessian of $\phi_k$ at $x_k$ and therefore vanishes by our assumptions \eqref{eq:phik_assumptions} on $\phi_k$, 
which also imply that $(t\mapsto F_t^{(k)}(x_k))'(0) = v_k$. Consequently, \eqref{eq:Dpp} gives 
\[
\frac{D''(0)}{D(0)} = -\frac{1}{n}\Ric_{g_k}(v_k,v_k),
\]
and substituting this into \eqref{eq:Cpp} proves \eqref{eq:Cdoubleprime}.

Finally, combining \eqref{eq:convexity_at_x_k} with \eqref{eq:Cdoubleprime} results in
\[
\Ric_{g_k}(v_k,v_k) = -\frac{\partial^2}{\partial t^2} C_k(x_k,0)  \ge \big(K-\frac{\delta}{2}\Big)g_k(v_k,v_k) > \Big(K - \delta\Big) g_k(v_k,v_k)
\]
for all $k\ge k_0$, giving the desired contradiction to \eqref{eq:indirect_assumption}. Thus we arrive at the following result:
\begin{Theorem}\label{th:synth_to_dist}
Let $M$ be a compact connected manifold with a $C^{1,1}$-Riemannian metric $g$ such that $(X,d_g,d\vol_g/\vol_g(M))$ has $\infty$-Ricci curvature bounded below by $K$.
Assume further that for some subsequence of $g\star_M \rho_\eps$, \eqref{eq:additional_assumption} is satisfied.
Then also $\Ric_g \ge Kg$ in the distributional sense. 
\end{Theorem}
\begin{remark}\label{rem:discussion} There is a large class of examples of $C^{1,1}$-Riemannian metrics that fail to be $C^2$ but still satisfy the additional
assumption \eqref{eq:additional_assumption} in Theorem \ref{th:synth_to_dist}. Indeed,
since the $g_k$ are obtained by convolution, for any $C^{1,1}$-metric that is $C^2$ outside a closed
zero set \eqref{eq:additional_assumption} clearly holds. In particular, this situation occurs whenever two $C^2$-Riemannian metrics are glued along a 
closed embedded submanifold of codimension greater or equal than $1$ in such a way that the resulting metric is $C^{1,1}$.
\end{remark}
To conclude this paper, let us mention some directions of further research that naturally suggest themselves based on the results derived above.
A main question is whether the optimal transport approach and the distributional method cease to produce equivalent lower Ricci curvature bounds when further
lowering the regularity of the metric, i.e., if the theories ``branch'' in the direction of lower differentiability. That this may indeed be the
case is supported by two observations: On the one hand, as has also become apparent in Section \ref{sec:synth_to_dist}, already for $C^{1,1}$
metrics one is deprived of many of the standard tools of Riemannian geometry and regularisation methods can only partially make up for this loss.
Whereas this may be seen as a mere technical inconvenience, it should be taken into account that geometric properties that are taken for granted 
for $C^2$-metrics in fact cease to hold below this regularity. 
As examples we mention the coming apart of the notions of local distance minimisers and geodesics (solutions of the geodesic equation)
for $C^{1,\alpha}$-metrics ($\alpha\in (0,1)$), cf.\ \cite{HW,SS18}, or the more obvious fact that geodesic branching is a generic
phenomenon for $C^1$-metrics. 
It is also an open question whether the well-known equivalence of various entropy conditions (e.g., that of $\mathrm{CD}(K,\infty)$ with the $\infty$-Ricci bounds
employed here) continues to hold below $C^2$,
i.e., if the synthetic approach itself might branch when lowering the differentiability class of the metric. 

The methods used in this paper are not specifically tied to metrics of Riemannian signature. In fact, our regularisation results were derived 
from constructions that had first been developed in the context of generalising the classical singularity theorems of general relativity, hence
can also be used to analyse metrics of Lorentzian (or indeed arbitrary) signature. Also in this direction, similar questions arise: It has
been noted in recent years that a number of standard results of Lorentzian causality theory may lose their validity for metrics of low regularity
(and certainly do so below the Lipschitz class), cf., e.g., the phenomenon of `bubbling' metrics (\cite{CG,GKSS}). The Lorentzian synthetic framework \cite{lls}
employed in \cite{CM22} avoids such pathologies explicitly, but at the prize of excluding certain continuous but non-Lipschitz spacetimes
from consideration. Finally, it will be of interest to compare 
the synthetic approach to generalising the singularity theorems of general relativity (\cite{CM22}) with the distributional one
(\cite{hawkingc11,penrosec11,GGKS,G20,KOSS22}).

\medskip\noindent
{\bf Acknowledgements.} We thank Christian Ketterer for helpful discussions. 
This work was supported by project P 33594 of the Austrian Science Fund FWF.

\end{document}